\documentclass{amsart}

\usepackage{tikz}
\usepackage{amsmath}
\usepackage{amsthm}
\usepackage{mathrsfs}
\usepackage{lmodern}
\usepackage[T1]{fontenc}
\usepackage[utf8]{inputenc}
\usepackage{afterpage}

\usepackage{hyperref}
\hypersetup{colorlinks=true,linkcolor=blue,urlcolor=blue,citecolor=blue}

\usepackage[backend=bibtex,style=alphabetic,firstinits,url=false,doi=false,maxbibnames=50,maxcitenames=50,maxnames=50]{biblatex}

\newtheorem{theorem}[equation]{Theorem}
\newtheorem{proposition}[equation]{Proposition}

\newtheorem{lemma}[equation]{Lemma}
\theoremstyle{definition}

\newcommand{\define}[1]{\textbf{#1}}
\DeclareMathOperator{\lp}{L}
\newcommand{\nbar}{\vert\!\vert}

\newcommand{\Fraczek}{Fr\k{a}czek}
\newcommand{\Lemanczyk}{Lema\'{n}czyk}

\title{Mild mixing of certain interval exchange transformations}
\author{Donald Robertson}

\begin{document}

\begin{abstract}
We prove that irreducible, linearly recurrent, type W interval exchange transformations are always mild mixing.
For every irreducible permutation the set of linearly recurrent interval exchange transformations has full Hausdorff dimension.
\end{abstract}

\maketitle

\section{Introduction}
\label{sec:introduction}

Fix a permutation $\pi$ of $\{1,\dots,d\}$ and positive lengths $\lambda_1,\dots,\lambda_d$ that sum to 1.
Put $\lambda_0 = 0$.
Write
\[
I_i = [\lambda_0 + \cdots + \lambda_{i-1},\lambda_0 + \cdots + \lambda_i)
\]
for each $1 \le i \le d$.
The \define{interval exchange transformation} on $[0,1)$ determined by the data $(\lambda,\pi)$ is the map $T : [0,1) \to [0,1)$ given by
\[
Tx
=
x - \sum_{j < i} \lambda_j + \sum_{\pi j < \pi i} \lambda_j
\]
for all $x$ in $I_i$.
Thus $T$ rearranges the intervals $I_1,\dots,I_d$ according to the permutation $\pi$.
Every interval exchange transformation preserves Lebesgue measure on $[0,1)$.
Katok~\cite{MR594335} showed that interval exchange transformations are never mixing.

If for some $1 \le k < d$ the set $\{1,\dots,k\}$ is $\pi$ invariant then any interval exchange transformation with permutation $\pi$ is the concatenation of interval exchange transformations on fewer intervals.
It is therefore typical to assume that no such $k$ exists.
In this case $\pi$ is said to be \define{irreducible}.
Veech~\cite{MR644019} and Masur~\cite[Theorem~1]{MR644018} showed independently that for any irreducible permutation $\pi$ and almost every choice of $\lambda_1,\dots,\lambda_d$ in the simplex $\lambda_1 + \cdots + \lambda_d = 1$ the corresponding interval exchange transformation is uniquely ergodic.
Veech~\cite[Theorem~1.4]{MR765582} proved the almost every interval exchange transformation is rigid.
Avila and Forni~\cite[Theorem~A]{MR2299743} proved that almost every interval exchange transformation is weak mixing whenever $\pi$ is not a rotation of $\{1,\dots,d\}$.

In this paper we prove that for any permutation in an infinite class introduced by Veech~\cite{MR765582} and studied by Chaves and Nogueira~\cite{MR1878073} under the moniker of type W permutations, the set of $\lambda$ for which the interval exchange transformation determined by $(\lambda,\pi)$ is mild mixing has full Haudsorff dimension.

\begin{theorem}
\label{thm:mainTheorem}
For every irreducible, type W permutation $\pi$ on $\{1, \dots, d \}$ the set of lengths $(\lambda_1,\dots,\lambda_d)$ for which $(\lambda,\pi)$ is mild mixing has full Hausdorff dimension.
\end{theorem}

In fact, we will prove in Section~\ref{sec:proof} that whenever $\pi$ is type W and $\lambda$ is such that $(\lambda,\pi)$ is linearly recurrent (see Section~\ref{sec:linearRecurrence}) then $(\lambda,\pi)$ is mild mixing.
It follows from work of Kleinbock and Weiss~\cite{MR2049831} that, for a fixed irreducible permutation $\pi$, the set of such $\lambda$ in the simplex has full Hausdorff dimension.
(See Theorem~\ref{thm:linRecWinning} below.)
Chaika, Cheung and Masur~\cite{MR3296560} have extended Klienbock and Weiss's result by showing that the set is winning for Schmidt's game, which implies full Hausdorff dimension.

For interval exchanges on three intervals Theorem~\ref{thm:mainTheorem} is a consequence of work by Ferenczi, Holton and Zamboni~\cite[Theorem~4.1]{MR2129107} who showed that, for such interval exchanges, linear recurrence implies minimal self-joinings and hence mild mixing.
Boshernitzan and Nogueira~\cite[Theorem~5.3]{MR2060994} showed that all interval exchange transformations that are type W and linearly recurrent are weakly mixing.
We remark that examples due to Himli~\cite{hmili} prove this is not the case in general: indeed taking $m = 4$ in the first example of Section~3 therein gives a permutation that is not type W and an interval exchange transformation that has a non-constant eigenfunction; one can verify that the interval exchange transformation is linearly recurrent provided $\alpha$ is badly approximable.

For flows on surfaces \Fraczek{} and \Lemanczyk{}~\cite{MR2237466} showed that special flows over irrational rotations with bounded partial quotients whose roof function is piecewise absolutely continuous and has non-zero sum of jumps is always mild mixing.
Subsequently, \Fraczek{}, \Lemanczyk{} and Lesigne~\cite{MR2342268} gave a criterion for a piecewise constant roof function over an irrational rotation with bounded partial quotients to be mild mixing.
Using this criterion \Fraczek{}~\cite{MR2538574} has shown that, when the genus is at least two, the set of Abelian differentials for which the vertical flow is mild mixing is dense in every stratum of moduli space.
More recently, Kanigowski and Ku\l{}aga-Przymus~\cite{kanigowskiKulaga2015} showed that roof functions over interval exchange transformation having symmetric logarithmic singularities at some of the discontinuities of the interval exchange transformation give rise to special flows that are mild mixing.

Sections~\ref{sec:typeW} and \ref{sec:linearRecurrence} contain the facts we need about type W and lienarly recurrent interval exchange transformations respectively.
In Section~\ref{sec:rigidity} we discuss rigidity and mild mixing.

We would like to thank Jon Chaika for suggesting this project and for his help and patience during conversations related to it.
The author gratefully acknowledges the support of the NSF via grant DMS-1246989.

\section{Type W permutations}
\label{sec:typeW}

Fix a permutation $\pi$ on $\{1,\dots,d\}$.
Define a permutation $\sigma$ of $\{ 0,\dots,d\}$ by
\[
\sigma(j)
=
\begin{cases}
\pi^{-1}(1) - 1 & j = 0 \\
d & \pi(j) = d \\
\pi^{-1}(\pi(j) + 1) - 1 & \textrm{otherwise}
\end{cases}
\]
and write $\Sigma(\pi)$ for the set of orbits of $\sigma$.
This auxiliary permutation was introduced by Veech~\cite{MR644019}.
It describes which intervals are adjacent after an application of any interval exchange transformation defined by the permutation $\pi$.

We can represent $\Sigma(\pi)$ as a directed graph with $\{0,\omega_1,\dots,\omega_{d-1},1\}$ as its set of vertices: writing $\omega_0 = 0$ and $\omega_d = 1$ there is an edge from $\omega_i$ to $\omega_j$ if and only if $\sigma(i) = j$.
The edge with source $0$ and the edge with target $1$ correspond, respectively, to the first two cases in the definition of $\sigma$.
Call this graph the \define{endpoint identification graph}.
Say that $\pi$ is \define{type W} if the vertices $0$ and $1$ are in distinct components of the graph.
Some examples of type W permutations are given in \cite[Section~5]{MR2504744}.

Define
\[
T_+(a) = \lim_{x \to a+} T(x)
\qquad
T_-(a) = \lim_{x \to a-} T(x)
\]
for all $a$ in $[0,1)$ and $(0,1]$ respectively.
The endpoint identification graph contains the edge $0 \to \omega_k$ if and only if $0 = T_+(\omega_k)$ and the edge $\omega_j \to 1$ if and only if $T_-(\omega_j) = 1$.
All other edges $\omega_j \to \omega_k$ correspond to equalities $T_-(\omega_j) = T_+(\omega_k)$ where $\omega_j \ne 0$ and $\omega_k \ne 1$.

\section{Linear recurrence}
\label{sec:linearRecurrence}

Fix an irreducible permutation $\pi$ on $\{1,\dots,d\}$.
Given lengths $\lambda_1,\dots,\lambda_d$ put $\beta_i = \lambda_1 + \cdots + \lambda_i$ for all $1 \le i \le d-1$.
Let $D = \{ \beta_1,\dots,\beta_{d-1} \}$.
One says that $(\lambda,\pi)$ satisfies the \define{infinite distinct orbits condition} if $D \cap T^{-n} D = \emptyset$ for all $n$ in $\mathbb{N}$.
Keane~\cite{MR0357739} showed that the infinite distinct orbits condition implies minimality.

Assuming the infinite distinct orbits condition, the set
\[
\bigcup \{ T^{-i} D : 0 \le i \le n \}
\]
partitions $[0,1)$ into sub-intervals of positive length.
Write $\epsilon_n$ for the length of the shortest interval in this partition.
It was observed in \cite{MR961737} that if $T$ satisfies the infinite distinct orbits condition and $J \subset [0,1)$ is an interval of length at most $\epsilon_n$ then there are times $p,q \ge 0$ with $p + q = n - 1$ such that the sets
\begin{equation}
\label{eqn:ietTower}
T^{-p} J, T^{-p+1}J,\dots,J,\dots,T^{q-1}J, T^q J
\end{equation}
are disjoint intervals.
Moreover $T^{k}J = T^{k-1}J + \alpha_k$ for all $-p < k \le q$, which is to say that each interval is a translate of the previous one under $T$.
We call \eqref{eqn:ietTower} the \define{tower} defined by $J$.
Call $T^{-p}J$ the \define{bottom floor} of the tower and $T^q J$ the \define{top floor} of the tower -- see Figure~\ref{fig:tower} for a schematic.
Write $\tau$ for the union of the sets in \eqref{eqn:ietTower}.
If $J$ contains a discontinuity of $T$ then $q = 0$, and if $J$ contains a discontinuity of $T^{-1}$ then $p = 0$.
Note that disjointness of the intervals implies $n \epsilon_n \le 1$ so we must have $\epsilon_n \to 0$ as $n \to \infty$.

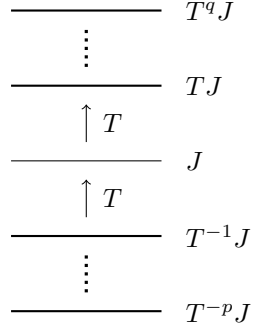
\begin{figure}
\centering
\begin{tikzpicture}
\draw[thick] (-1,0)--(1,0);
\node[anchor=west] at (1.2,0) {$T^{-p} J$};
\draw[thick] (-1,1)--(1,1);
\node[anchor=west] at (1.2,1) {$T^{-1} J$};
\draw (-1,2)--(1,2);
\node[anchor=west] at (1.2,2) {$J$};
\draw[thick] (-1,3)--(1,3);
\node[anchor=west] at (1.2,3) {$T J$};
\draw[thick] (-1,4)--(1,4);
\node[anchor=west] at (1.2,4) {$T^{q} J$};
\draw[->] (0,1.25) -- (0,1.75);
\node[anchor=west] at (0.1,1.5) {$T$};
\draw[->] (0,2.25) -- (0,2.75);
\node[anchor=west] at (0.1,2.5) {$T$};
\draw[dotted,very thick] (0,0.25)--(0,0.75);
\draw[dotted,very thick] (0,3.25)--(0,3.75);
\end{tikzpicture}
\caption{The tower determined by the interval $J$. Each interval is a translate of the one below under $T$.}
\label{fig:tower}
\end{figure}

An interval exchange transformation is \define{linearly recurrent} if
\[
\inf \{ n \epsilon_n : n \in \mathbb{N} \} \ge c
\]
for some positive constant $c$.
Thus for linearly recurrent interval exchange transformations the tower \eqref{eqn:ietTower} determined by any interval $J$ of length $\epsilon_n$ has measure at least $c$.

Fix an irreducible permutation $\pi$.
We conclude this section with a proof of the following result.

\begin{theorem}
\label{thm:linRecWinning}
For every irreducible permutation $\pi$ the set of $\lambda$ for which the interval exchange transformation $(\lambda,\pi)$ is linearly recurrent is strong winning.
\end{theorem}

An interval exchange transformation $T$ is \define{badly approximable} if
\[
\inf \{ n |q - T^n(p)| : n \in \mathbb{N} \} > 0
\]
holds for all discontinuities $p,q$ of $T$.
Note that badly approximable implies the infinite distinct orbits condition.
It follows from the proof of Theorem~1.4 in \cite{MR3296560} that those $\lambda$ for which $(\lambda,\pi)$ is badly approximable is strong winning and in particular has full Hausdorff dimension.
To prove Theorem~\ref{thm:linRecWinning} it therefore suffices to prove that every badly approximable interval exchange transformation is linearly recurrent.
We give a proof of this folklore result for completion.

\begin{proposition}
Fix an irreducible permutation $\pi$.
If for some $\lambda$ the interval exchange transformation defined by $(\lambda,\pi)$ is badly approximable then it is linearly recurrent.
\end{proposition}
\begin{proof}
Fix an interval exchange transformation $T$ that is badly approximable.
Let $c$ be the minimum of the quantities
\[
\inf \{ n |q - T^n(p)| : n \in \mathbb{N} \}
\]
over all discontinuities $p,q$ of $T$.
By hypothesis $c$ is positive.
Let $\eta$ be the minimal spacing between discontinuities of $T$.

Suppose that $T$ is not linearly recurrent.
Then $n \epsilon_n \le c/2$ and $\epsilon_n < \eta$ for some $n \ge 2$.
Fix $0 \le l \le m \le n$ and discontinuities $p,q$ of $T$ such that $\epsilon_n = |T^l(q) - T^m(p)|$.
We claim that if $1 \le l$ there is no discontinuity of $T^{-1}$ between $T^m(p)$ and $T^l(q)$.
Indeed, all discontinuities of $T^{-1}$ are of the form $T^i(r)$ for some $i \in \{1,2\}$ and some discontinuity $r$ of $T$, so the existence of a discontinuity of $T^{-1}$ between $T^m(p)$ and $T^l(q)$ contradicts $\epsilon_n = |T^l(q) - T^m(p)|$.
Thus $|T^{l-1}(q) - T^{m-1}(p)| = \epsilon_n$.
Iterating gives $|q - T^{m-l}(p)| = \epsilon_n$.
Since $\epsilon_n < \eta$ we must have $0 < m-l$.
But then
\[
c \le \inf \{ n |q - T^n(p)| : n \in \mathbb{N} \} \le (m-l)|q - T^{m-l}(p)| = (m-l) \epsilon_n \le n \epsilon_n \le c/2
\]
which is absurd.
\end{proof}

\section{Rigidity}
\label{sec:rigidity}

A measure-preserving transformation $T$ on a probability space $(X,\mathscr{B},\mu)$ is \define{rigid} if there is a sequence $i \mapsto n_i$ in $\mathbb{N}$ with $n_i \to\infty$ such that for every $f$ in $\lp^2(X,\mathscr{B},\mu)$ one has $T^{n_i} f \to f$ in $\lp^2(X,\mathscr{B},\mu)$.
We remark (see, for instance \cite[Section~2]{MR3255429}) that if a measure-preserving transformation $T$ on a probability space $(X,\mathscr{B},\mu)$ has the property that for every $f$ in $\lp^2(X,\mathscr{B},\mu)$ there is a a sequence $i \mapsto n_i$ such that $T^{n_i} f \to f$ in $\lp^2(X,\mathscr{B},\mu)$ then it is rigid.

\begin{lemma}
\label{lem:rigidityChar}
Let $T$ be a rigid, measure-preserving transformation on a probability space $([0,1),\mathscr{B},\mu)$ where $X = [0,1)$ and $\mu$ is Lebesgue measure.
Then $T$ is rigid if and only if there is a sequence $n_i \to \infty$ such that
\begin{equation}
\label{eqn:rigidityChar}
\mu( \{ x \in [0,1) : |T^{n_i} x - x| > \epsilon \} ) \to 0
\end{equation}
for every $\epsilon > 0$.
\end{lemma}
\begin{proof}
If $T$ is rigid then $\nbar T^{n_i} \iota - \iota \nbar \to 0$ in $\lp^2(X,\mathscr{B},\mu)$ and \eqref{eqn:rigidityChar} holds by Chebychev's inequality.
Conversely, one can use \eqref{eqn:rigidityChar} to prove that $\nbar T^{n_i} f - f \nbar \to 0$ for all continuous functions $f$ on $[0,1)$.
\end{proof}

A measure-preserving transformation $T$ is \define{mildly mixing} if it has no non-trivial rigid factors.
One can show that this is equivalent to the absence of non-constant functions $f$ in $\lp^2(X,\mathscr{B},\mu)$ such that $T^{n_i} f \to f$ in $\lp^2(X,\mathscr{B},\mu)$.
Indeed, given a particular sequence $i \mapsto n_i$ the subspace
\[
\{ f \in \lp^2(X,\mathscr{B},\mu) : T^{n_i} f \to f \}
\]
can be shown to be of the form $\lp^2(X,\mathscr{C},\mu)$ for some $T$ invariant sub-$\sigma$-algebra $\mathscr{C}$, and the corresponding factor is non-trivial if the above subspace contains non-constant functions.
Conversely, any rigid function on a factor lifts to a rigid function on the original system.

\section{Proof of main theorem}
\label{sec:proof}

In this section we will prove Theorem~\ref{thm:mainTheorem}.
We begin with the following lemma.

\begin{lemma}
\label{lem:quantitaveErgodicity}
Let $T$ be an ergodic, measure-preserving transformation on a probability space $(X,\mathscr{B},\mu)$.
If, given $f$ in $\lp^2(X,\mathscr{B},\mu)$, one can find a constant $\rho > 0$ such that
\[
\mu ( \{ x \in X : |f(T^{i+1} x) - f(T^i x)| < \delta \textrm{ for all } -b \le i \le b \} ) \ge \rho
\]
for all $\delta > 0$ and all $b \in \mathbb{N}$ then $f$ is constant.
\end{lemma}
\begin{proof}
For each $\delta > 0$ the sets
\[
\{ x \in X : |f(T^{i+1} x) - f(T^i x)| < \delta \textrm{ for all } -b \le i \le b \}
\]
are decreasing as $b \to \infty$ so their intersection has measure at least $\rho$ by hypothesis.
This intersection is $T$ invariant so has full measure by ergodicity.
In particular $\{ x \in X : |f(Tx) - f(x)| < \delta \}$ has full measure.
Since $\delta > 0$ is arbitrary $f$ is $T$ invariant almost surely and therefore constant by ergodicity.
\end{proof}

By Theorem~\ref{thm:linRecWinning} the following result implies Theorem~\ref{thm:mainTheorem}.

\begin{theorem}
A linearly recurrent, type W interval exchange transformation must be mildly mixing.
\end{theorem}
\begin{proof}
Fix a type W interval exchange transformation $T$ on $[0,1)$ that is linearly recurrent.
Let
\[
0 \to \zeta_1 \to \cdots \to \zeta_s \to 0
\]
be the loop in the endpoint identification graph that contains $0$.
Put $c = \inf \{ n \epsilon_n : n \in \mathbb{N} \}$.
Assume that for some $f$ in $\lp^2(X,\mathscr{B},\mu)$ there is a sequence $n_i \to \infty$ such that $T^{n_i} f \to f$ in $\lp^2(X,\mathscr{B},\mu)$.
Write $f_l$ for $f \circ T^l$.
We show that for every $\delta > 0$ and every $b \in \mathbb{N}$ the set
\[
\{ x \in [0,1) : |f_i(T x) - f_i(x)| < \delta \textrm{ for all } -b \le i \le b \}
\]
has measure at least $\frac{c}{10}$.
It will then follows from Lemma~\ref{lem:quantitaveErgodicity} and ergodicity of all linearly recurrent, type W interval exchange transformations (\cite[Theorem~5.2]{MR2060994}) that $f$ is constant.
Thus the only rigid functions in $\lp^2(X,\mathscr{B},\mu)$ are the constant functions and $T$ is mildly mixing.

Fix $\delta > 0$ and $b \in \mathbb{N}$.
Using Lusin's theorem one can find a compact set $K \subset [0,1)$ with $\mu(K) \ge 1 - \frac{c}{300 s}$ on which each of $T^{-b} f,\dots,f,\dots,T^b f$ is uniformly continuous.
Fix $\eta > 0$ so small that whenever $x,y \in K$ with $|x-y| < \eta$ one has $|f(T^i x) - f(T^i y)| < \frac{\delta}{8s}$ for all $-b \le i \le b$.
Fix using Lemma~\ref{lem:rigidityChar} and the fact that $\epsilon_n \to 0$ a time $n \in \mathbb{N}$ such that
\begin{enumerate}
\item
$\epsilon_n < \min \{ \frac{c}{200}, \eta, \frac{\epsilon_1}{4} \}$;
\item
the set
\[
G = \{ x \in[0,1) : |f_i(T^n x) - f_i(x)| < \tfrac{\delta}{8s} \textrm{  for all } -b \le i \le b \}
\]
has measure at least $1 - \frac{c}{300 s}$.
\end{enumerate}
Put $H = G \cap K \cap T^{-n} K$.
We have $\mu(H) \ge 1 - \frac{c}{100 s}$.

For each $1 \le i \le s$ let $I_i$ be the interval $[\zeta_i - \frac{\epsilon_n}{2}, \zeta_i + \frac{\epsilon_n}{2})$.
As described in Section~\ref{sec:linearRecurrence} each $I_i$ determines a tower
\[
T^{-(n-1)} I_i, T^{-(n-2)} I_i, \dots, T^{-1} I_i, I_i
\]
of disjoint intervals with total measure at least $c$.
Write $\tau_i$ for this tower.
Put $I_0 = [0,\frac{\epsilon_n}{2})$.
Similarly, it is the roof of a tower $\tau_0$ with total measure at least $\frac{c}{2}$.
We claim that $90\%$ of the points $x$ in $\tau_0$ have all the following properties:
\begin{enumerate}
\item
$x \in T^{-\ell} I_0$ with $\ell \ne n-1$;
\item
$x \in H$ and $T^{-1} x \in H$;
\item
for every $1 \le i \le s$ one can find points $y_i$ in $H \cap T^{-\ell} [\zeta_i - \frac{\epsilon_n}{2},\zeta_i)$ and $z_i$ in $H \cap T^{-\ell}[\zeta_i,\zeta_i + \frac{\epsilon_n}{2})$.
\end{enumerate}
This follows from the following arguments.
\begin{enumerate}
\item
$99\%$ of the points in $\tau_0$ are not in the bottom level since $100 \epsilon_n < c/2$.
\item
$\mu(\tau_0 \cap H) \ge \frac{99c}{100s}$ and $\mu(\tau_0 \cap T^{-1} H) \ge \frac{99c}{100s}$ so $\mu(\tau_0 \cap H \cap T^{-1} H) \ge \frac{98c}{100s}$.
Thus $98\%$ of the points in $\tau_0$ satisfy Property 2.
\item
Suppose one can find $r$ distinct levels $T^{-l_1}I_0,\dots,T^{-l_r}I_0$ in $\tau_0$ and for each such level a ``defective'' tower $\tau_{m_i}$ with either $H \cap T^{-l_i} [\zeta_{m_i} - \frac{\epsilon_n}{2},\zeta_{m_i}) = \emptyset$ or $H \cap T^{-l_i}[\zeta_{m_i}, \zeta_{m_i} + \frac{\epsilon_n}{2}) = \emptyset$.
By pigeonhole at least $r/s$ of these levels have the same defective tower, say $\tau_m$.
But then $\mu(\tau_m \setminus H) \ge \tfrac{r}{s} \tfrac{\epsilon_n}{2} \ge \tfrac{r}{s} \tfrac{c}{2n}$.
But $\mu(\tau_m \setminus H) \le \frac{c}{100s}$ so $\tfrac{r}{n} \le 2\%$.
Thus $97\%$ of the points in $\tau_0$ enjoy Property 3.
\end{enumerate}
All told, at least $90\%$ of the points in $\tau_0$ satisfy all three properties.

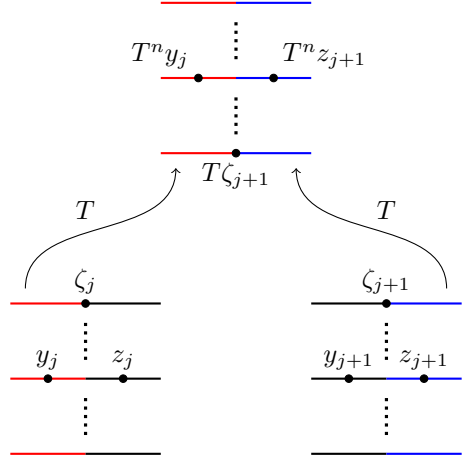
\begin{figure}[t]
\centering
\begin{tikzpicture}
\draw[red,thick] (-3,0)--(-2,0);
\draw[red,thick] (-3,1)--(-2,1);
\draw[red,thick] (-3,2)--(-2,2);
\draw[thick] (-2,0)--(-1,0);
\draw[thick] (-2,1)--(-1,1);
\draw[thick] (-2,2)--(-1,2);
\draw[thick] (1,0)--(2,0);
\draw[thick] (1,1)--(2,1);
\draw[thick] (1,2)--(2,2);
\draw[blue,thick] (2,0)--(3,0);
\draw[blue,thick] (2,1)--(3,1);
\draw[blue,thick] (2,2)--(3,2);
\draw[red,thick] (-1,4)--(0,4);
\draw[red,thick] (-1,5)--(0,5);
\draw[red,thick] (-1,6)--(0,6);
\draw[blue,thick] (0,4)--(1,4);
\draw[blue,thick] (0,5)--(1,5);
\draw[blue,thick] (0,6)--(1,6);
\draw[->] (-2.8,2.2) to [out=90,in=270] (-0.8,3.8);
\node[anchor=south] at (-2,3) {$T$};
\draw[->] (2.8,2.2) to [out=90,in=270] (0.8,3.8);
\node[anchor=south] at (2,3) {$T$};
\node[anchor=south] at (-2.5,1) {$y_j$};
\draw[fill] (-2.5,1) circle [radius=0.05];
\node[anchor=south] at (-1.5,1) {$z_j$};
\draw[fill] (-1.5,1) circle [radius=0.05];
\node[anchor=south] at (1.5,1) {$y_{j+1}$};
\draw[fill] (1.5,1) circle [radius=0.05];
\node[anchor=south] at (2.5,1) {$z_{j+1}$};
\draw[fill] (2.5,1) circle [radius=0.05];
\node[anchor=south east] at (-0.5,5) {$T^n y_j$};
\draw[fill] (-0.5,5) circle [radius=0.05];
\node[anchor=south west] at (0.5,5) {$T^n z_{j+1}$};
\draw[fill] (0.5,5) circle [radius=0.05];
\node[anchor=south] at (-2,2) {$\zeta_j$};
\draw[fill] (-2,2) circle [radius=0.05];
\node[anchor=south] at (2,2) {$\zeta_{j+1}$};
\draw[fill] (2,2) circle [radius=0.05];
\node[anchor=north] at (0,4) {$T\zeta_{j+1}$};
\draw[fill] (0,4) circle [radius=0.05];
\draw[dotted,very thick] (-2,0.25)--(-2,0.75);
\draw[dotted,very thick] (-2,1.25)--(-2,1.75);
\draw[dotted,very thick] (2,0.25)--(2,0.75);
\draw[dotted,very thick] (2,1.25)--(2,1.75);
\draw[dotted,very thick] (0,4.25)--(0,4.75);
\draw[dotted,very thick] (0,5.25)--(0,5.75);
\end{tikzpicture}
\caption{The towers and their relationships for an edge $\zeta_j \to \zeta_{j+1}$ in the endpoint identification graph with $0 \notin \{ \zeta_j,\zeta_{j+1} \}$.}
\label{fig:regularEdgeTowers}
\end{figure}

Fix a point $x$ in $\tau_0$ such that all three properties hold.
In particular, let $y_1,\dots,y_s,z_1,\dots,z_s$ and $\ell$ be as in Property 3.
For every $1 \le j \le s - 1$ and every edge $\zeta_j \to \zeta_{j+1}$ we have the estimates
\begin{equation}
\label{eqn:rigidtyRegularEdge}
|f_i(y_j) - f_i(T^n y_j)| < \tfrac{\delta}{8s}
\qquad
|f_i(T^n z_{j+1}) - f_i(z_{j+1})| < \tfrac{\delta}{8s}
\end{equation}
for all $-b \le i \le b$ since both $y_j$ and $z_{j+1}$ belong to $G$.
We also have
\[
|z_{j+1} - y_{j+1}| < \eta
\]
since $\{ y_{j+1}, z_{j+1} \} \subset T^{-\ell} [\zeta_{j+1} - \frac{\epsilon_n}{2}, \zeta_{j+1} + \frac{\epsilon_n}{2})$ and
\[
|T^n y_j - T^n z_{j+1}| < \eta
\]
since $[T \zeta_{j+1} - \frac{\epsilon_n}{2}, T \zeta_{j+1} + \frac{\epsilon_n}{2})$ contains $\{ T^{\ell + 1} y_j, T^{\ell+1} z_{j+1} \}$ and is the bottom floor of a tower of width $\epsilon_n$.
(See Figure~\ref{fig:regularEdgeTowers} for a schematic.)
These two inequalities imply
\[
|f_i(z_{j+1}) - f_i(y_{j+1})| < \tfrac{\delta}{8s}
\qquad
|f_i(T^n y_j) - f_i(T^n z_{j+1})| < \tfrac{\delta}{8s}
\]
for all $-b \le i \le b$ since $\{ y_j, y_{j+1}, z_{j+1} \} \subset K \cap T^{-n} K$.
Combined with \eqref{eqn:rigidtyRegularEdge} these give
\begin{equation}
\label{eqn:regularEdges}
|f_i(y_j) - f_i(y_{j+1})| < \tfrac{\delta}{2s}
\end{equation}
for all $-b \le i \le b$.

\afterpage{
\begin{figure}[t]
\centering
\begin{tikzpicture}
\draw[thick,blue] (-4,0)--(-3,0);
\draw[thick,blue] (-4,1)--(-3,1);
\draw[thick,blue] (-4,2)--(-3,2);
\draw[thick] (-3,0)--(-2,0);
\draw[thick] (-3,1)--(-2,1);
\draw[thick] (-3,2)--(-2,2);
\draw[thick] (2,0)--(3,0);
\draw[thick] (2,1)--(3,1);
\draw[thick] (2,2)--(3,2);
\draw[thick,red] (3,0)--(4,0);
\draw[thick,red] (3,1)--(4,1);
\draw[thick,red] (3,2)--(4,2);
\draw[thick] (-0.5,0)--(0.5,0);
\draw[thick] (-0.5,1)--(0.5,1);
\draw[thick,red] (-0.5,2)--(0.5,2);
\draw[blue,thick] (-1,4)--(0,4);
\draw[blue,thick] (-1,5)--(0,5);
\draw[blue,thick] (-1,6)--(0,6);
\draw[blue,thick] (-1,7)--(0,7);
\draw[red,thick] (0,4)--(1,4);
\draw[red,thick] (0,5)--(1,5);
\draw[red,thick] (0,6)--(1,6);
\draw[red,thick] (0,7)--(1,7);
%
\draw[->] (-3.8,2.2) to [out=90,in=270] (-0.8,3.8);
\node[anchor=south] at (-3,3) {$T$};
\draw[->] (0.3,2.2) to [out=90,in=270] (0.8,3.8);
\node[anchor=west] at (0.7,3) {$T$};
\draw[->] (3.8,2.2) to [out=90,in=0] (3,3) to [out=180,in=270] (0.3,1.8);
\node[anchor=south] at (3,3) {$T$};
%
\node[anchor=south] at (-3.5,1) {$y_s$};
\draw[fill] (-3.5,1) circle [radius=0.05];
\node[anchor=south] at (-2.5,1) {$z_s$};
\draw[fill] (-2.5,1) circle [radius=0.05];
\node[anchor=south] at (-3,2) {$\zeta_s$};
\draw[fill] (-3,2) circle [radius=0.05];
\node[anchor=south] at (2.5,1) {$y_1$};
\draw[fill] (2.5,1) circle [radius=0.05];
\node[anchor=south] at (3.5,1) {$z_1$};
\draw[fill] (3.5,1) circle [radius=0.05];
\node[anchor=south] at (3,2) {$\zeta_1$};
\draw[fill] (3,2) circle [radius=0.05];
\node[anchor=south west] at (0.3,1) {$x$};
\draw[fill] (0.3,1) circle [radius=0.05];
\node[anchor=south] at (-0.5,2) {$0=T\zeta_1$};
\draw[fill] (-0.5,2) circle [radius=0.05];
\node[anchor=south east] at (-0.5,6) {$T^n y_s$};
\draw[fill] (-0.5,6) circle [radius=0.05];
\node[anchor=south west] at (0.8,6) {$T^n x$};
\draw[fill] (0.8,6) circle [radius=0.05];
\node[anchor=south west] at (0.5,5) {$T^n z_1$};
\draw[fill] (0.5,5) circle [radius=0.05];
\node[anchor=north] at (0,4) {$T0$};
\draw[fill] (0,4) circle [radius=0.05];
\draw[dotted,very thick] (-3,0.25)--(-3,0.75);
\draw[dotted,very thick] (-3,1.25)--(-3,1.75);
\draw[dotted,very thick] (0,0.25)--(0,0.75);
\draw[dotted,very thick] (0,1.25)--(0,1.75);
\draw[dotted,very thick] (3,0.25)--(3,0.75);
\draw[dotted,very thick] (3,1.25)--(3,1.75);
\draw[dotted,very thick] (0,4.25)--(0,4.75);
\draw[dotted,very thick] (0,6.25)--(0,6.75);
\end{tikzpicture}
\caption{The towers and their relationships for the edges $0 \to \zeta_1$ and $\zeta_s \to 0$ in the endpoint identification graph. Note that $T^n z_1$ and $T^{n-1} x$ are in the same level since $T \zeta_1 = 0$.}
\label{fig:specialEdgeTowers}
\end{figure}
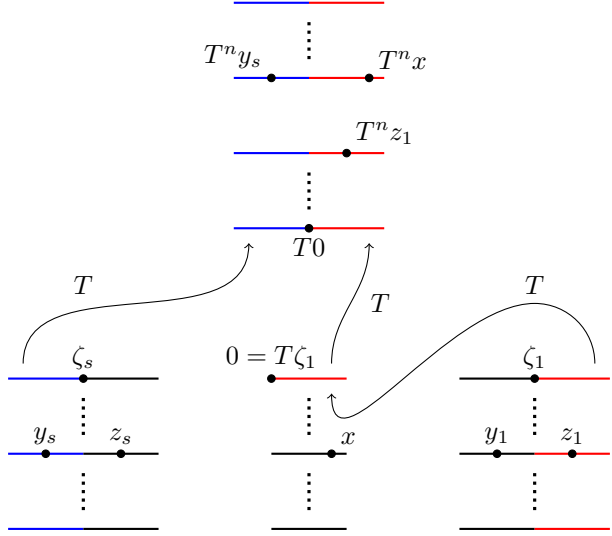
}

Considering next the edge $\zeta_s \to 0$, we have the estimates
\begin{equation}
\label{eqn:rigidtyIntoZero}
|f_i(x) - f_i(T^n x)| < \tfrac{\delta}{8s}
\qquad
|f_i(T^n y_s) - f_i(y_s)| < \tfrac{\delta}{8s}
\end{equation}
for all $-b \le i \le b$ since $\{ x, y_s \} \in G$.
Also
\[
|T^n x - T^n y_s| < \eta
\]
since $[T(0) - \frac{\epsilon_n}{2}, T(0) + \frac{\epsilon_n}{2})$ contains $\{ T^{\ell + 1} x, T^{\ell + 1} y_s\}$ and is the bottom floor of a tower of width $\epsilon_n$.
(See Figure~\ref{fig:specialEdgeTowers} for a schematic.)
Thus
\[
|f_i(T^n x) - f_i(T^n  y_s)| \le \tfrac{\delta}{8s}
\]
for all $-b \le i \le b$ since $\{x,y_s\} \subset T^{-n} K$.
Combined with \eqref{eqn:rigidtyIntoZero} we get
\begin{equation}
\label{eqn:intoZero}
|f_i(x) - f_i(y_s)| \le \tfrac{\delta}{2s}
\end{equation}
for all $-b \le i \le b$.

Using \eqref{eqn:regularEdges} for all $1 \le j \le s-1$ and together with \eqref{eqn:intoZero} yields
\begin{equation}
\label{eqn:almostThere}
|f_i(x) - f_i(y_1)| < \tfrac{\delta}{2}
\end{equation}
for all $-b \le i \le b$.
Now we use the edge $0 \to \zeta_1$ to pick up some invariance.
First, note that
\begin{equation}
\label{eqn:rigidityOutofZero}
|f_i(z_1) - f_i(T^n z_1)| < \tfrac{\delta}{8s}
\qquad
|f_i(T^{n}(T^{-1} x)) - f_i(T^{-1} x)| < \tfrac{\delta}{8s}
\end{equation}
for all $-b \le i \le b$ since $\{ z_1, T^{-1} x \} \subset H$.
We also have
\[
|y_1 - z_1| < \eta
\]
because $\{y_1,z_1\} \subset T^{-\ell} [\zeta_1 - \frac{\epsilon_n}{2},\zeta_1 + \frac{\epsilon_n}{2})$ and
\[
|T^n z_1 - T^n(T^{-1} x)| < \eta
\]
because $T^n z_1$ and $T^n(T^{-1} x)$ are both in the interval $T^{n-\ell-1} [0,\tfrac{\epsilon_n}{2})$.
(See Figure~\ref{fig:specialEdgeTowers} for a schematic.)
These two estimates imply
\[
|f_i(y_1) - f_i(z_1)| < \tfrac{\delta}{8s}
\qquad
|f_i(T^n z_1) - f_i(T^n(T^{-1} x))| < \tfrac{\delta}{8s}
\]
for all $-b \le i \le b$ because $\{y_1,z_1,T^n(T^{-1} x), T^n z_1 \} \subset K$.
Together with \eqref{eqn:rigidityOutofZero} these imply
\[
|f_i(y_1) - f_i(T^{-1} x)| < \tfrac{\delta}{2s}
\]
for all $-b \le i \le b$.
Finally, combined with \eqref{eqn:almostThere} we get
\[
|f_i(x) - f_i(T^{-1} x)| < \delta
\]
for all $-b \le i \le b$.

Since $x$ was an arbitrary point in $\tau_0$ satisfying Properties 1, 2 and 3, the set
\[
\{ x \in [0,1) : |f_i(x) - f_i(T^{-1} x)| < \delta \textrm{ for all }  -b \le i \le b \}
\]
has measure at least $\tfrac{c}{10}$.
Since $\delta > 0$ and $b \in \mathbb{N}$ were arbitrary, the function $f$ is constant by Lemma~\ref{lem:quantitaveErgodicity}.
\end{proof}

\printbibliography

\end{document}